\documentclass{article}
\usepackage[utf8]{inputenc}
\usepackage{amssymb, amsmath}
\usepackage{amsthm}
\usepackage{amsfonts} 
\usepackage{tikz-cd}

\usetikzlibrary{arrows}

\title{Some Identities in Quantum Torus Arising from Ringel-Hall Algebras}
\author{Jiuzhao Hua} 

\newtheorem{thm}{Theorem}[section]
\newtheorem{lem}{Lemma}[section]
\newtheorem{cor}{Corollary}[section]

\usepackage[parfill]{parskip}

\begin{document}
\date{\vspace{-0.5cm}}\date{} 
\maketitle\begin{abstract}
We define two classes of representations of quivers over arbitrary fields, called monomorphic representations and epimorphic representations. We show that every representation has a unique maximal nilpotent subrepresentation and the associated quotient is always monomorphic, and every representation has a unique maximal epimorphic subrepresentation and the associated quotient is always nilpotent. The uniquenesses of such subrepresenations imply two identities in the Ringel-Hall algebra. By applying Reineke's integration map, we
obtain two identities in the corresponding quantum torus.
\end{abstract}

\section{Introduction}

Let $\mathbb{N}$ be the set of all non-negative integers, $\mathbb{Z}$ the ring of integers and $\mathbb{Q}$ the field of rational numbers. 
Let $Q = (Q_0, Q_1)$ be a finite quiver and assume that $Q$ is connected and
$Q_0 = \{1,2,\cdots,n\}$ with $n\ge 1$. For any arrow $h\in Q_1$, let $h'$ and $h''$ be its \textit{source vertex} and \textit{target vertex} respectively,
depicted as $_{h'}\bullet \overset{h}{\longrightarrow} \bullet_{h''}$. 
The \textit{Euler form} associated with $Q$ is defined as follows:

\begin{equation}
\langle v,w \rangle := \sum_{i\in Q_0}v_iw_i - \sum_{h\in Q_1} v_{h'}w_{h''} \,\,\text{ for } v,w\in \mathbb{Z}^n.
\end{equation}

A \textit{representation} of $Q$ over a field $k$ is a collection of finite dimensional $k$-vector spaces and $k$-linear maps 
$M = (V_i, f_h)_{ i \in Q_0, h\in Q_1}$ such that $f_h$ is a $k$-linear map
from $V_{h'}$ to $V_{h''}$. $(\dim_k V_1, \cdots, \dim_k V_n)\in\mathbb{N}^n$ is called the \textit{dimension vector} of $M$, denoted by $\dim M$.
A \textit{morphism} from $M$ to another representation $N = (W_i, g_h)_{ i \in Q_0, h\in Q_1}$ is a collection of $k$-linear maps $(\phi_i: V_i \to W_i)_{ i \in Q_0}$ such that 
$\phi_{h''} f_h = g_h\phi_{h'}$ for all $h\in Q_1$. If $\phi_i$ is non-singular for all $i\in Q_0$, then $M$ and $N$ are called \textit{isomorphic}. $N$
is called a \textit{subrepresentation} of $M$ if $W_i\subset V_i$ for all $i\in Q_0$ and $g_h(x) = f_h(x)$ for all $h\in Q_1$ and $x\in W_{h'}$. In this case we write $g _h = f_h|_{W_{h'}}$.

A vertex $i\in Q_0$ is called a \textit{source vertex} if it is not a target of any arrow in $Q_1$, and $i$ is called a \textit{sink vertex} if it is not a source of any arrow in $Q_1$.
We assume that $Q$ has no sink vertices and no source vertices. If $Q$ does have sink vertices or source vertices, we will work over an extended quiver $\overline{Q}$, 
which is constructed as follows:
\begin{enumerate}
\item add a new vertex labelled as $n+1$,
\item add a new arrow from vertex $n+1$ to each source vertex,
\item add a new arrow from each sink vertex to vertex $n+1$.
\end{enumerate}
In this setting, every representation of $Q$ can be extended to a representation of $\overline{Q}$ by adding the zero vector space to vertex $n+1$ 
and the zero map to every arrow connected to vertex $n+1$.

Let $N= (W_i, f_h|W_{h'})_{ i \in Q_0, h\in Q_1}$ be a subrepresentation of $M = (V_i, f_h)_{ i \in Q_0, h\in Q_1}$, we define an operator $im^-$ as follows:
$$
im^-(N) := (U_i, f_h|{U_{h'}})_{ i \in Q_0, h\in Q_1} \text{ where } U_i = \bigcap_{h'=i}f^{-1}_h(W_{h''}) \subset V_i \text{ for } i\in Q_0.
$$
Moreover, we let
$$
im^{-1}(N) = im^-(N) \text{ and } im^{-i}(N) = im^-(im^{-(i-1)}(N)) \text{ for } i > 1.
$$
Note that $im^-(N)$ is always a subrepresentation of $M$ and $N \subset im^-(N)$. Thus we have the following chain of subrepresentations of $M$:
\begin{equation}\label{filtration -}
0 \subset im^{-1}(0) \subset  im^{-2}(0)  \subset \cdots \subset  im^{-s}(0)  \subset \cdots
\end{equation}

$M$ is called a \textit{nilpotent representation} if $im^{-s}(M) = M$ for some $s\in\mathbb{N}$ and
$M$ is called a \textit{monomorphic representation} if $im^{-1}(0)=0$. 
Note that $M$ is nilpotent if and only there exists $s\in\mathbb{N}$ such that each vector space $V_i$ (for $i\in Q_0$) is mapped
to the zero vector space by any path starting from vertex $i$ as long as the length of the path is greater $s$. Thus, the
definition of nilpotency here is equivalent to the one given by Bozec, Schiffmann \& Vasserot \cite{B-S-V 2018}. 
Note that if $Q$ has no oriented cycles, then the only nilpotent representation of $Q$ is the zero representation.

We define another operator $im^+$ as follows:
$$
im^+(N) := (U_i, f_h|U_{h'})_{ i \in Q_0, h\in Q_1} \text{ where } U_i = \sum_{h''=i}f_h(W_{h'}) \subset W_i \text{ for } i\in Q_0.
$$
Moreover, we let
$$
im^{+1}(N) = im^+(N) \text{ and } im^{+i}(N) = im^+(im^{+(i-1)}(N)) \text{ for } i > 1.
$$
Note that $im^+(N)$ is always a subrepresentation of $M$ and $im^+(N)\subset N$. Thus we have the following chain of subrepresentations of $M$:
\begin{equation}\label{filtration +}
M \supset im^{+1}(M) \supset im^{+2}(M) \supset \cdots \supset im^{+s}(M) \supset \cdots
\end{equation}

$M$ is called an \textit{epimorphic represenation} if $im^{+1}(M) = M$. Note that $M$ is nilpotent if there exists $s\in\mathbb{N}$ such that
$im^{+s} (M)= 0$.

Let $\mathrm{mod}_{k}(Q)$ be the category of representations of $Q$ over $k$ and let $\mathrm{mod}^n_{k}(Q)$ (resp. $\mathrm{mod}^m_{k}(Q)$, $\mathrm{mod}^e_{k}(Q)$)
be the subcategory consisting of all nilpotent (resp. monomorphic, epimorphic) representations. All three subcategories are closed under 
direct summands and extensions. Let $\underline{\mathrm{mod}}_{k}(Q)$ (resp. $\underline{\mathrm{mod}}^n_{k}(Q)$, $\underline{\mathrm{mod}}^m_{k}(Q)$, $\underline{\mathrm{mod}}^e_{k}(Q)$) denote the set of isomorphism classes in $\mathrm{mod}_{k}(Q)$(resp. $\mathrm{mod}^n_{k}(Q)$, $\mathrm{mod}^m_{k}(Q)$,
$\mathrm{mod}^e_{k}(Q)$). The isomorphism class of an object $M$ in a category is denoted by $[M]$ and the cardinality of a set $X$ is denoted by $|X|$.

For any finite field $k$, Ringel \cite{CR 1990} defines an associative algebra $\mathcal{H}(Q)$ known as the \textit{Ringel-Hall algebra} associated with $Q$. As a $\mathbb{Q}$-vector space, 
$\mathcal{H}(Q)$ has a basis $\{[M] \,|\, [M] \in \underline{\mathrm{mod}}_{k}(Q)\}$,
and the multiplication of two basis elements is given by:
$$
[M]\circ [N]: = \sum_{[X]} g^X_{MN} [X],
$$
where $g^X_{MN} $ is the number of subrepresentations $U$ of $X$ such that $U\cong N$ and $X/U \cong M$.

Ringel \cite{CR 1990} proved that when $Q$ is of finite representation type (precisely when the underlying graph of $Q$ is a Dynkin diagram by Gabriel \cite{PG 1972}) the twisted Ringel-Hall
algebra is isomorphic to the quantized enveloping algebra of the positive part of the semisimple Lie algebra associated with $Q$. 
For general quivers, Green \cite{JG 1995} proved that the composition subalgebra of the twisted
Ringel-Hall algebra is isomorphic to the quantized enveloping algebra of the positive  part of the Kac-Moody algebra associated with $Q$.

Let $X_1,\cdots,X_n$ be $n$ non-commuting indeterminates, and let $X^v = \prod_{i=1}^nX_i^{v_i}$ for any vector $v=(v_1,\cdots,v_n)\in\mathbb{N}^n$.
Let $\mathcal{T}(Q)$ be the \textit{torus algebra} associated with  $Q$. Thus as a  $\mathbb{Q}$-vector space, $\mathcal{T}(Q)$ has a basis $\{X^v |\, v\in \mathbb{N}^n\}$,
and the multiplication of two basis elements is defined by the following rule:
\begin{equation}
X^v \circ X^w = q^{ -\langle v, w \rangle} X^{v+w} \, \text{ for all } v,w\in \mathbb{N}^n.
\end{equation}

Reineke's integration map is a $\mathbb{Q}$-linear map from $\mathcal{H}(Q)$ to $\mathcal{T}(Q)$ defined as follows:
\begin{equation}
\delta : [M] \,\,\to\,\, \frac{1}{|\mathrm{Aut}(M)|} X^{\dim M},
\end{equation}
where $\mathrm{Aut}(M)$ is the automorphism group of $M$.
The Riedtmann-Peng formula (\cite{CRD 1994}\cite{LP 1998}) can be stated as:
$$
g^X_{MN} = \frac{|\mathrm{Aut}(X)|\cdot |\mathrm{Ext}^1(M,N)_X|}
{|\mathrm{Hom}(M,N)|\cdot |\mathrm{Aut}(M)|\cdot |\mathrm{Aut}(N)|},
$$
where $\mathrm{Ext}^1(M,N)_X$ is the set of extension classes corresponding to short exact sequences with middle term isomorphic to $X$. The
Riedtmann-Peng formula implies that Reineke's integration map $\delta$ is a homomorphism of algebras (Lemma 3.3 of Reineke \cite{MR 2006}).
Reineke's integration map can be naturally extended to a homomorphism from the completion algebra of $\mathcal{H}(Q)$ to 
the completion algebra of $\mathcal{T}(Q)$. In these two completion algebras, the sum and product of two infinite series are well defined because
 both algebras are graded by dimension vectors.

Let $\mathbb{Q}(t)$ be the field of rational functions in $t$ over $\mathbb{Q}$ and let $\mathcal{T}_t(Q)$ be the \textit{quantum torus} over $\mathbb{Q}(t)$ associated with $Q$. 
Thus as a $\mathbb{Q}(t)$-vector space, $\mathcal{T}_t(Q)$ has a basis $\{X^v | \,v\in \mathbb{N}^n\}$ and the multiplication of two basis elements is defined by the following rule:
\begin{equation}
X^v \circ X^w = t^{ -\langle v, w \rangle} X^{v+w} \, \text{ for all } v,w\in \mathbb{N}^n.
\end{equation}

In the next section, we show that each representation of $Q$ is a unique extension of a monomorphic representation by a nilpotent representation, and it is also 
a unique extension of a nilpotent representation by an epimorphic representation. The uniquenesses of those extensions yield two identities in the Ringel-Hall algebra
involving the numbers of representations of $Q$ over finite fields. 
These identities are then translated into identities in the quantum torus by Reineke's integration map. 
In the last section we present two identities which count the isomorphism classes of absolutely simple (resp. indecomposable) 
conservative representations of $Q$ over finite fields.

\section{Two Identities in Quantum Torus}

Each representation $M = (V_i, f_h)_{ i \in Q_0, h\in Q_1}$  gives rise to two linear maps $\sigma_i$ and $\tau_i$ for each $i\in Q_0$ as follows:
\begin{align*}
\sigma_i :  V_i & \,\,\to\,\, \bigoplus_{h'=i} V_{h''}, & {} \!\!\!\!\!\!\tau_i: \bigoplus_{h''=i}V_{h'} & \,\,\to\,\, V_i , \\
v_i & \,\,\mapsto\,\, (f_h(v_i))_{h'=i},  & {} (v_{h'}) & \,\,\mapsto\,\, \sum_{h''=i} f_h(v_{h'}). 
\end{align*}

The following lemma is a direct consequence of the definitions of monomorphic and epimorphic representations.
\begin{lem}
Let $M = (V_i, f_h)_{ i \in Q_0, h\in Q_1}$ be a representation of $Q$. Then we have
\begin{enumerate}
\item $M$ is monomorphic if and only $\sigma_i$ is injective for all $i\in Q_0$.
\item $M$ is epimorphic if and only $\tau_i$ is surjective for all $i\in Q_0$,
\end{enumerate}
\end{lem}

\begin{lem}\label{lem monomorphic}
For any representation $M$ of $Q$ over a field $k$, there exists a unique maximal nilpotent subrepresentation $N\subset M$.  Moreover, $M/N$ is monomorphic.
\end{lem}
\begin{proof}
Since $M$ is finite dimensional, chain (\ref {filtration -}) must be stable in finite steps, i.e., there exists an integer $s$ such that $im^{-t}(0) = im^{-s}(0)$ for all $t>s$. 
We claim that $im^{-s}(0)$ is the unique maximal nilpotent subrepresentation of $M$. 

$im^{-s}(0)$ is obviously nilpotent. Let $N=(W_i, f_h|_{W_{h'}})_{ i \in Q_0, h\in Q_1}$ be any nilpotent subrepresentation of $M$, we claim that $N\subset  im^{-s}(0)$. Since $N$ is nilpotent, there exists $r\in\mathbb{N}$ 
such that each vector space $W_i$ (for $i\in Q_0$) is mapped to the zero vector space by any path starting from vertex $i$ as long as the length of the path is greater than $r$. This
implies that $N\subset im^{-(s+r)}(0)$. Since $im^{-(s+t)}(0) = im^{-s}(0)$, we have $N\subset im^{-s}(0)$.

To show that $M/im^{-s}(0)$ is monomorphic, we only need to show the $\sigma_i$ is injective for each $i\in Q_0$ when $\sigma_i $ is acting on vector spaces in $M/im^{-s}(0)$, which is an easy consequence of the definition of $im^{-}$.

\end{proof}

\begin{lem}\label{lem epimorphic}
For any representation $M$ of $Q$ over a field $k$, there exists a unique maximal epimorphic subrepresentation $E\subset M$.  Moreover, $M/E$ is nilpotent.
\end{lem}
\begin{proof}
Since $M$ is finite dimensional, chain (\ref {filtration +}) must be stable in finite steps, i.e., there exists an integer $s$ such that $im^{+t}(M) = im^{+s}(M)$ for all $t>s$. 
We claim that $im^{+s}(M)$ is the unique maximal epimorphic subrepresentation of $M$. 

$im^{+s}(M)$ is obviously epimorphic. Let $E$ be any epimorphic subrepresentation of $M$, then $im^{+1}(E) =E$ and hence $im^{+s}(E) = E$. $E \subset M$ implies that 
$im^{+s}(E) \subset im^{+s}(M)$ and hence $E\subset im^{+s}(M)$. Thus $im^{+s}(M)$ is the unique maximal epimorphic subrepresentation. 
$M/im^{+s}(M)$ is nilpotent because $im^{+i}(M)/im^{+(i+1)}(M)$ is nilpotent for all $i\ge 1$.
\end{proof}

In what follows, let $k$ be the finite field $\mathbb{F}_q$ with $q$ elements, where $q$ is a prime power. Lemma \ref{lem monomorphic} and Lemma \ref{lem epimorphic} implies the following theorem.
\begin{thm}\label{prod_in_hall_alg}
The following identities hold in the completion algebra of the Ringel-Hall algebra $\mathcal{H}(Q)$:
\begin{align*}
\sum_{[M]\in\underline{\mathrm{mod}}_{k}(Q)}[M] &= \Bigg( \sum_{[M]\in\underline{\mathrm{mod}}_{k}^m(Q)}[M] \Bigg)\circ 
\Bigg( \sum_{[M]\in\underline{\mathrm{mod}}_{k}^n(Q)}[M] \Bigg),\\
\sum_{[M]\in\underline{\mathrm{mod}}_{k}(Q)}[M] &= \Bigg( \sum_{[M]\in\underline{\mathrm{mod}}_{k}^n(Q)}[M] \Bigg)\circ 
\Bigg( \sum_{[M]\in\underline{\mathrm{mod}}_{k}^e(Q)}[M] \Bigg).
\end{align*}
\end{thm}

Applying Reineke's integration map $\delta$ on both sides of the identities in Theorem \ref{prod_in_hall_alg}, it yields the following results.

\begin{cor}\label{prod_in_torus}
The following identities hold in the completion algebra of the torus algebra $\mathcal{T}(Q)$:
\begin{align*}
\sum_{v\in\mathbb{N}^n} \frac{|\mathrm{R}(v,\mathbb{F}_q)|}{|\mathrm{GL}(v, \mathbb{F}_q)|} X^v &= \left(\sum_{v\in\mathbb{N}^n} \frac{|\mathrm{M}(v,\mathbb{F}_q)|}{|\mathrm{GL}(v, \mathbb{F}_q)|} X^v\right) \circ
 \left(\sum_{v\in\mathbb{N}^n} \frac{|\mathrm{N}(v,\mathbb{F}_q)|}{|\mathrm{GL}(v, \mathbb{F}_q)|} X^v\right),\\
\sum_{v\in\mathbb{N}^n} \frac{|\mathrm{R}(v,\mathbb{F}_q)|}{|\mathrm{GL}(v, \mathbb{F}_q)|} X^v &= \left(\sum_{v\in\mathbb{N}^n} \frac{|\mathrm{N}(v,\mathbb{F}_q)|}{|\mathrm{GL}(v, \mathbb{F}_q)|} X^v\right) \circ
 \left(\sum_{v\in\mathbb{N}^n} \frac{|\mathrm{E}(v,\mathbb{F}_q)|}{|\mathrm{GL}(v, \mathbb{F}_q)|} X^v\right).
\end{align*}
\end{cor}

If $Q$ has source vertices or sink vertices, then Corollary \ref{prod_in_torus} holds for $\overline Q$. By taking $X^{n+1} = 0$ we recover those two identities for $Q$.

For any dimension vector $v=(v_1,\cdots,v_n)\in\mathbb{N}^n$, let
$v\cdot v = \sum_{i=1}^n v_i^2\in\mathbb{N}$, 
$v_\bullet = (v_\bullet^i)\in\mathbb{N}^n$ with $v_\bullet^i = \sum_{h'=i}v_{h''}$, 
$_\bullet v= (_\bullet v^i)\in\mathbb{N}^n$ with $_\bullet v^i = \sum_{h''=i}v_{h'}$, and 
$\mathrm{GL}(v,\mathbb{F}_q) = \prod_{i=1}^n\mathrm{GL}(v_i,\mathbb{F}_q)$, where $\mathrm{GL}(v_i,\mathbb{F}_q)$ is the General Linear Group
of order $v_i$ over $\mathbb{F}_q$. Let $\mathit{gl}(s,t) = \prod_{i=0}^{s-1}(t^s-t^i)\in\mathbb{Q}(t)$ for $s\in\mathbb{N}$ and $\mathit{gl}(v,t) = \prod_{i=1}^n gl(v_i,t)$, then we have 
$|\mathrm{GL}(v,\mathbb{F}_q)| = \mathit{gl}(v,q)$.

For any given dimension vector $v\in\mathbb{N}^n$, let $\mathrm{R}(v,\mathbb{F}_q)$ (resp. $\mathrm{N}(v,\mathbb{F}_q)$, $\mathrm{M}(v,\mathbb{F}_q)$, $\mathrm{E}(v,\mathbb{F}_q)$) 
denote the set of all representations (resp. nilpotent, monomorphic and epimorphic representations)  of $Q$ over $\mathbb{F}_q$ with dimension vector $v$.

Define the following rational functions in $\mathbb{Q}(t)$:
\begin{align*}
&\mathit{r}(v,t) = t^{v\cdot v -\langle v, v \rangle}, \\
&\mathit{m}(v,t) = t^{\langle v, v \rangle}  \frac{\mathit{gl}(v_{\bullet}, t)}{\mathit{gl}(v_{\bullet}-v, t)} \,\,\text{ if } v_{\bullet}\ge v \text{, else } \mathit{m}(v,t) = 0, \\
&\mathit{e}(v,t) = t^{\langle v, v \rangle}  \frac{\mathit{gl}(_{\bullet}v, t)}{\mathit{gl}(_{\bullet}v-v, t)} \,\, \text{ if } _{\bullet}v\ge v \text{, else } \mathit{e}(v,t) = 0.
\end{align*}
Then we have $|\mathrm{R}(v,\mathbb{F}_q)| = \mathit{r}(v,q)$, $|\mathrm{M}(v,\mathbb{F}_q)| = \mathit{m}(v,q)$ and $|\mathrm{E}(v,\mathbb{F}_q)| = \mathit{e}(v,q)$.

Define the rational function $\mathit{n}(v,t)$ as follows:
\begin{equation}\label{num_of_N}
\frac{\mathit{n}(v,t)} {\mathit{gl}(v,t)} = \sum_{v^*}t^{-\sum_{k<l}\langle v^{(k)}, v^{(l)} \rangle} \prod_{k\ge 1} 
\frac{H(v^{(k)}, v^{(k+1)}, t^{-1})} {\mathit{gl}(v^{(k)},t)}, 
\end{equation}
where the sum ranges over all tuples $v^* = (v^{(1)}, v^{(2)}, \dots)$ of non-zero dimension vectors such that $\sum_{k\ge 1} v^{(k)}= v$
and $_{\bullet}v^{(k)} \ge v^{(k+1)}$ for $k\ge 1$, and where the function $H$ is defined by:
\begin{equation}\label{num_of_H}
H(v,w,t^{-1}) = t^{(_{\bullet}v-w)\cdot (_{\bullet}v-w) - _{\bullet}v\cdot _{\bullet}v} \frac{\mathit{gl}(_{\bullet}v, t)}
{\mathit{gl}(_{\bullet}v - w, t)}.
\end{equation}

Then we  have $|\mathrm{N}(v,\mathbb{F}_q)| = \mathit{n}(v,q)$, which is due to Bozec, Schiffmann \& Vasserot \cite{B-S-V 2018}.

For the Jordan quiver i.e., the quiver with one vertex and one loop, the identities in Corollary \ref{prod_in_torus} are reduced to the following identity which is trivial:
\begin{equation}\label{prod_for_Jordan_quiver}
\sum_{n=0}^\infty \frac{q^{n^2} X^n}{\prod_{i=0}^{n-1}(q^n-q^i)} = \left(\sum_{n=0}^\infty \frac{q^{n^2-n}X^n}{\prod_{i=0}^{n-1}(q^n-q^i)}\right) \cdot
 \left(\sum_{n=0}^\infty X^n\right).
\end{equation}

For the wild quiver which has one vertex and 2 loops, the first few terms of the identities in Corollary \ref{prod_in_torus} are as follows:
\begin{align*}
&R(1,q) = q^2, \\
&R(2,q) = q^8, \\
&R(3,q) = q^{18}, \\
&N(1,q) = 1, \\
&N(2,q) = q^3+q^2-q, \\
&N(3,q) = q^9+2q^8-q^6-2q^5+q^3, \\
&M(1,q) = q^2-1, \\
&M(2,q) = q^8-q^5-q^4+q, \\
&M(3,q) = q^{18}-q^{14}-q^{13}-q^{12}+q^9+q^8+q^7-q^3, \\
&E(1,q) = q^2-1, \\
&E(2,q) = q^8-q^5-q^4+q, \\
&E(3,q) = q^{18}-q^{14}-q^{13}-q^{12}+q^9+q^8+q^7-q^3. \\
\end{align*}
Because $|\mathrm{R}(v,\mathbb{F}_q)|$, $|\mathrm{N}(v,\mathbb{F}_q)|$, $|\mathrm{E}(v,\mathbb{F}_q)|$, $|\mathrm{M}(v,\mathbb{F}_q)|$ 
and $|\mathrm{GL}(v,\mathbb{F}_q)|$
are all polynomial functions in $q$ with rational coefficients and there are infinitely many prime numbers, the two identities in Corollary \ref{prod_in_torus} must also hold in the quantum torus.

\begin{cor} The following identities hold in the completion algebra of the quantum torus $\mathcal{T}_t(Q)$:
\begin{align*}
\sum_{v\in\mathbb{N}^n} \frac{r(v, t)}{{gl}(v, t)} X^v &= \left(\sum_{v\in\mathbb{N}^n} \frac{m(v, t)}{{gl}(v, t)} X^v\right) \circ
 \left(\sum_{v\in\mathbb{N}^n} \frac{n(v, t)}{{gl}(v, t)} X^v\right),\\
\sum_{v\in\mathbb{N}^n} \frac{r(v, t)}{{gl}(v, t)} X^v &= \left(\sum_{v\in\mathbb{N}^n} \frac{n(v, t)}{{gl}(v, t)} X^v\right) \circ
 \left(\sum_{v\in\mathbb{N}^n} \frac{e(v, t)}{{gl}(v, t)} X^v\right).
\end{align*}
\end{cor}

\section{Numbers of absolutely indecomposable conservative representations}
A representation of $Q$ is called \textit{conservative} if it is monomorphic and epimorphic at the same time. The subcategory of conservative representations is denoted by
$\mathrm{mod}_k^c(Q)$. It is closed under direct summands and extensions. A representation of $Q$ over a field is called \textit{absolutely indecomposable} if it is still indecomposable 
when the ground field is extended to its algebraic closure. 
Let $c(v,q)$ be the number of conservative representations of $Q$ over $\mathbb{F}_q$ with 
dimension $v\in\mathbb{N}^n$ and let $a(v,q)$ (resp. $s(v,q)$) 
be the number of isomorphism classes of absolutely indecomposable conservative (resp. absolutely simple conservative) representations of $Q$ over $\mathbb{F}_q$ with dimension $v$.

Following the methods of Mozgovoy and Reineke \cite{MR 2006}\cite{MR 2009}, we have the following result.

\begin{cor}\label{simple_consev}
Assuming that $c(v,q)$ is a polynomial function in $q$ with rational coefficients for all $v\in\mathbb{N}^n$, we have the following
identity in the completion algebra of the quantum torus $\mathcal{T}_q(Q)$:
\begin{equation}
\Bigg(\sum_{v\in\mathbb{N}^n} \frac { c(v,q) } {\mathit{gl}(v,q) } X^v\Bigg) \!\circ 
\mathrm{Exp}  \Bigg(\sum_{v\in\mathbb{N}^n\backslash\{0\}} \!\frac{s(v,q)}{1-q} X^v\Bigg) = 1,\\
\end{equation}
where $\mathrm{Exp}$ is the plethystic exponential map (see \cite{MR 2009} for details).
\end{cor}

Following the methods of Hua \cite{JH 2021}, we have the following result.
\begin{cor}\label{abs_indcomp}
Assuming that $c(v,q)$ is a polynomial function in $q$ with rational coefficients for all $v\in\mathbb{N}^n$, we have the following identity in the 
formal power series ring $\mathbb{Q}(q)[[X_1,\cdots,X_n]]$:
\begin{equation}
1 + \sum_{v^*} \!
\Bigg(\prod_{s\ge 1} \frac{q^{\langle v^s\!,\,v^s\rangle}} {q^{\langle \beta^s\!,\, \beta^s \rangle}}  \frac{c(v^s, q)}{\mathit{gl}(v^s, q)} X^{sv^s}\Bigg) =
 \,\mathrm{Exp}\Bigg(\sum_{v\in\mathbb{N}^n\backslash\{0\}}\! \frac{a(v,q)} {q-1}X^v\Bigg),
\end{equation}
where the sum runs over all tuples of dimension vectors  $v^*=(v^1,v^2,\cdots, v^r)$ such that $r\ge 1$, $v^i\in\mathbb{N}^n$ for $1\le i \le r$ and $v^r\ne 0$, 
and $\beta^s=\sum_{i\ge s} v^i$ for $1\le s \le r$.
\end{cor}

Consequently, $s(v,q)$ and $a(v,q)$ are polynomials in $q$ with rational coefficients. Elementary calculations suggest that their coefficients are always integers.
Using the classification of indecomposable representations of the Kronecker quiver, one can verify that the assumptions in Corollary \ref{simple_consev} and \ref{abs_indcomp} are ture
for the wild quiver which has one vertex and 2 loops.

\vspace{0.2cm}
\textbf{\large{Acknowledgments}}
\vspace{0.1cm}

I would like to thank Yingbo Zhang, Bangming Deng and Xueqing Chen for their valuable discussions and comments.

\vspace{0.2cm}

\textit{Email address}: \texttt{jiuzhao.hua@gmail.com}

\begin{thebibliography}{100}
\bibitem{B-S-V 2018}T. Bozec, O. Schiffmann \& E. Vasserot, \emph{On the numbers of points of nilpotent quiver varieties over finite fields}, 
 Annales scientifiques de l’Ecole normale supérieure, 2020, ff10.24033/asens.2452ff. ffhal-01709055.
\bibitem{PG 1972}P. Gabriel, \emph{Unzerlegbare Darstellungen. I }, Manuscripta Math. 6 (1972) 71–103; correction, ibid. 6, 309 (1972).
\bibitem{JG 1995}J. A. Green, \emph{Hall algebras, hereditary algebras and quantum groups}, Invent. Math. 120, 361-377 (1995).
\bibitem{JH 2021} J. Hua, \emph{On Numbers of Semistable Representations of Quivers over Finite Fields}, arXiv:2109.08905 (2021).
\bibitem{VK 1983}V. G. Kac, \emph{Root Systems, Representations of Quivers and Invariant Theory}, Lecture Notes in Mathematics 996, pp. 74–108, Springer-Verlag, Berlin (1983).
\bibitem{LP 1998}L. Peng, \emph{Lie algebras determined by finite Auslander-Reiten quivers}, Comm. in Alg., 26, no.9, 2711–2725 (1998).
\bibitem{MR 2009}S. Mozgovoy, \& M. Reineke, \emph{On the number of stable quiver representations over finite fields}, Journal of Pure and Applied Algebra 213, 430-439 (2009)
\bibitem{MR 2006}M. Reineke, \emph{Counting Rational Points of Quiver Moduli}, International Mathematical Research Notices 2006, pp. 1-19 (2006).
\bibitem{CRD 1994}C. Riedtmann, \emph{Lie algebras generated by indecomposables}, Journal of Algebra 170, no. 2, pp. 526–546 (1994).
\bibitem{CR 1990}C. M. Ringel,  \emph{Hall algebras and quantum groups}, Inv. Math. 101, pp. 583-591 (1990).

\end{thebibliography}
\end{document}